\documentclass[a4paper,12pt]{article}

\usepackage{tikz}
\usetikzlibrary{intersections,calc,arrows.meta}
\usetikzlibrary{arrows.meta,decorations.pathmorphing,backgrounds,positioning,fit,petri,through}
\usetikzlibrary{datavisualization}
\usetikzlibrary{decorations.pathreplacing}
\usetikzlibrary{angles,quotes}

\usepackage{amssymb} 
\usepackage{amsthm}
\usepackage{amsmath}
\usepackage{latexsym}
\usepackage{mathrsfs}
\usepackage{enumerate}
\usepackage{graphicx}
\usepackage{fancybox}
\usepackage{color}
\usepackage{multicol, multienum}
\usepackage{ulem}
\usepackage{bm}

\usepackage{mathtools}
\mathtoolsset{showonlyrefs=true}

\allowdisplaybreaks

\newtheorem{theorem}{Theorem}[section]
\newtheorem{lemma}[theorem]{Lemma}
\newtheorem{definition}[theorem]{Definition}

\newtheorem{remark}[theorem]{Remark}
\newtheorem{proposition}[theorem]{Proposition}
\newtheorem{corollary}[theorem]{Corollary}

\newcommand{\ma}{\color{magenta}}
\definecolor{mgre}{cmyk}{0.92,0.00,0.59,0.25}

\definecolor{PineGreen}{cmyk}{0.92,0.00,0.59,0.25}
\definecolor{ForestGreen}{cmyk}{0.91,0.00,0.88,0.12}
\definecolor{RawSienna}{cmyk}{0.00,0.72,1.00,0.45}
\definecolor{Mulbery}{cmyk}{0.34,0.90,0.00,0.02}
\definecolor{Sepia}{cmyk}{0.00,0.83,1.00,0.70}
\definecolor{Mahogany}{cmyk}{0.00,0.85,0.87,0.35}

\def\bE{{\mathbb{E}}}

\def\bM{{\mathbb{M}}}
\def\bN{{\mathbb{N}}}

\def\bP{{\mathbb{P}}}

\def\bR{{\mathbb{R}}}

\def\bfA{{\mathbf{A}}}

\def\bfJ{{\mathbf{J}}}
\def\bfK{{\mathbf{K}}}

\def\bfM{{\mathbf{M}}}

\def\bfY{{\mathbf{Y}}}

\def\cB{{\mathcal{B}}}

\def\cD{{\mathcal{D}}}
\def\cE{{\mathcal{E}}}
\def\cF{{\mathcal{F}}}

\def\cH{{\mathcal{H}}}

\def\cO{{\mathcal{O}}}
\def\cP{{\mathcal{P}}}

\def\1{{\mathbf{1}}}

\newcommand{\lon}{\Longleftrightarrow}

\newtheorem{thm}{Theorem}[section]

\newtheorem{prop}[thm]{Proposition}
\newtheorem{lem}[thm]{Lemma}
\newtheorem{cor}[thm]{Corollary}

\theoremstyle{definition}
\newtheorem{defi}[thm]{Definition}
\newtheorem{prob}{}[section]
\newtheorem{rem}[thm]{Remark}
\newtheorem{ex}[thm]{Example}
\newtheorem{pro}{Problem}[section]
\newtheorem{ass}[thm]{Assumption}

\newcommand{\capa}{\mathrm{Cap}}

\newcommand{\ov}{\overline}

\newcommand{\bd}{\begin{defi}}
\newcommand{\ed}{\end{defi}}

\newcommand{\bpro}{\begin{pro}}
\newcommand{\epro}{\end{pro}}

\newcommand{\bec}{\begin{cases}}
\newcommand{\eec}{\end{cases}}

\newcommand{\bpr}{\begin{prob}}
\newcommand{\epr}{\end{prob}}

\newcommand{\bt}{\begin{thm}}
\newcommand{\et}{\end{thm}}

\newcommand{\ba}{\begin{ass}}
\newcommand{\ea}{\end{ass}}

\newcommand{\br}{\begin{rem}}
\newcommand{\er}{\end{rem}}

\newcommand{\bpm}{\begin{pmatrix}}
\newcommand{\epm}{\end{pmatrix}}

\newcommand{\be}{\begin{ex}}
\newcommand{\ee}{\end{ex}}

\newcommand{\bp}{\begin{prop}}
\newcommand{\ep}{\end{prop}}

\newcommand{\bl}{\begin{lem}}
\newcommand{\el}{\end{lem}}

\newcommand{\bc}{\begin{cor}}
\newcommand{\ec}{\end{cor}}

\newcommand{\bq}{\begin{que}}
\newcommand{\eq}{\end{que}}

\newcommand{\beqn}{\begin{eqnarray*}}
\newcommand{\eeqn}{\end{eqnarray*}}

\newcommand{\beqnn}{\begin{eqnarray}}
\newcommand{\eeqnn}{\end{eqnarray}}

\newcommand{\bequ}{\begin{equation}}
\newcommand{\eequ}{\end{equation}}

\newcommand{\benu}{\begin{enumerate}}
\newcommand{\eenu}{\end{enumerate}}

\newcommand{\barr}{\begin{array}{rcl}}
\newcommand{\ear}{\end{array}}

\newcommand{\la}{\label}

\newcommand{\ds}{\displaystyle}

\newcommand{\sm}{\setminus}

\newcommand{\ttr}{\textrm}

\title{Harmonic functions for recurrent
symmetric $\alpha$-stable processeses with non-local perturbations}
\author{
\textsc{Kaneharu Tsuchida} 
}
\date{\today}
 
\begin{document}

\maketitle

  \begin{abstract}
   Let $\bfM$ be the recurrent symmetric (relativistic) $\alpha$-stable
   process on $\bR^d$. Let $\cH^{\mu + F} (:= \cH + \mu + F)$ be
   a Schr\"odinger type operator with local and non-local perturbations
   $\mu$ and $F$. 
   If $\mu$ and $F$ satisfy suitable conditions
   associated with Kato class, 
   we prove the existence of ground state for a Schr\"odinger type
   operator relating to $\cH^{\mu,F}$. 
   Furthermore, we prove the ground state becomes a
   probabilistically harmonic function
   of the Schr\"odinger operator generated by $\cH^{\mu,F}$. 
  \end{abstract}
  \bigskip

   \medskip
  
  \noindent
  {\bf AMS 2020 \textit{Mathematics Subject Classification}}.
  Primary 60J45; Secondary 60J76, 31C25.

  \section{Introduction}
 Let $\bfM = (\bP_x,X_t)$ be the symmetric (relativistic)
 $\alpha$-stable process on $\bR^d$ and $\cH^{}$ its generator,
 that is, $\cH^{} = m - (-\Delta + m^{2/\alpha})^{\alpha/2}\ $ 
 ($m \ge 0,\ 0 < \alpha < 2$, and $m = 0$ in standard stable case).
 It is well known that $\bfM$ is recurrent if
 $d = 1,2$ and $m > 0$, or $1 = d \le \alpha$ and $m = 0$.
 Let $(\cE^{}, \cF^{})$ be the recurrent Dirichlet form
 associated with $\bfM$. 
 Let $\mu$ be a signed smooth Green-tight Kato measure generated by $\bfM$. 
 In \cite{Ta}, Takeda showed the existence and
 continuity of the ground state associated
 with a Schr\"odinger type operator
 $\cH^{\mu} = -\cH^{} + \mu$ if $m = 0$. 
 In \cite{Tsu2}, we extend results in \cite{Ta} to the case $m > 0$.
 Since the potential term is a signed smooth measure,
 the Schr\"odinger operators in \cite{Ta, Tsu2} correspond to
 local Feynman-Kac perturbations.
 
 In this paper, we extend the arguments in \cite{Ta, Tsu2}
 to Schr\"odinger-type operators
 with non-local perturbations.
 Let $F(x,y)$ be a symmetric, bounded and
 measurable function on $\bR^d \times \bR^d$
 with $F(x,x) = 0$ for any $x \in \bR^d$.
 Then, the Schr\"odinger operator with
 non-local perturbation is defined by 
 \begin{align}
  \cH^{\mu,F}f(x) := \cH f(x) + \mu f(x) + \int_{\bR^d} \left(e^{F(x,y)} - 1\right)
  f(y) N(x,dy) \quad \text{for $f \in \cB_b(\bR^d)$},
  \la{Sch-01}
 \end{align}
 where $N(x,dy)$ is the jump kernel of L\'evy system
 associated with $\bfM$. 
 Let $(\cE^{\mu,F}, \cF)$ be the quadratic form associated with
 $\cH^{\mu,F}$ (see \eqref{p-form} below).
 Let $\{A_t\}$ be an additive functional defined by
 \begin{align}
  A_t := A_t^{\mu} + A_t^F = A_t^{\mu} + \sum_{s \le t}F(X_{s-}, X_s),
 \end{align}
 where $A_t^{\mu}$ is the continuous additive functional
 which corresponds to the local perturbation $\mu$. 
 Then it is known that the Feynman-Kac semigroup of $\cH^{\mu,F}$ is
 \begin{align}
  p_t^{\mu,F} f(x) = \bE_x\left[\exp\left(A_t\right) f(X_t)\right]
 \end{align}
 for $f$ is a nonnegative Borel function on $\bR^d$. 
 Noting the definition in \eqref{Sch-01},
 we would like to emphasis that 
 the decomposition $F^{+} - F^{-}$ of $F$
 does not imply the sign decomposition of the perturbation term
 in contrast to the local perturbations. 
 Due to this, considering non-local perturbations is more difficult than
 dealing with local ones. 
 A decomposition developed by Li \cite{Li}
 serves as a powerful tool in the analysis of
 non-local perturbations. 
 The decomposition is given as follows:
 \begin{align}
  G^+ = (e^{F^+} - 1) e^{-F^-}, \quad G^- = 1 - e^{-F^-}.
  \la{decomp-01}
 \end{align}
 Then it holds that
 \begin{align}
  G(x,y) := e^{F(x,y)} -1 = G^+(x,y) - G^-(x,y). 
 \end{align}
 While the negative part $G^-$ is fully determined by $F^-$,
 the positive part $G^+$ depends on both $F^{+}$ and $F^-$.
 Nevertheless, this decomposition is very convenient for
 analyzing non-local perturbations.
 
 Now, we briefly explain assumptions on perturbations.
 Suppose that all parts of perturbation belongs the Kato class.
 In this paper,  the Kato class is defined by means of additive functionals.
 Here we note the equivalence defined by the additive functional
 (that is, probabilistic definition) and
 (compensated) Green function (that is, classical or analytic definition).
 Especially, Kuwae and Takahashi \cite{KT} established
 this equivalence in low dimensional cases (recurrent cases). 
 Since $\bfM$ is recurrent, the normal Green function of $\bM$ does
 not exist. 
 Thus it is impossible to define the Green-tight Kato measure by
 using the Green function. In fact,
 for the transience case, the Green function is often used
 to define the Green-tightness. 
 Now we first kill $\bfM$ by $-\mu^- - F^-$.
 We denote by $\bfM^-$ the killed process.
 The process $\bfM^-$ is transient, so there exists the Green function.
 In this paper, we define the Green-tightness
 by using the Green function of $\bfM^-$.
 For a non-local perturbation $F$,
 we define (Green-tight) Kato classes as follows.
 Let $\xi_{F}$ be the Revuz measure corresponding to the dual
 predictable projection of purely discontinuous
 additive functional $A_t^{F} = \sum_{s \le t}F(X_{s-},X_s)$. 
 We define that $F$ belongs to a (Green-tight) Kato class
 by means of the measure $\xi_{F}$, similarly to the case of $\mu$. 

 Next we explain a ground state in this paper.
 Let $\xi_{G^+}$ be defined analogously to $\xi_F$,
 where $G^+$ is given in \eqref{decomp-01}.
 Then we put $\rho^+ = \mu^+ + \xi_{G^+}$, and
 a quadratic form $\cE^Y$ is defined by
 \begin{align}
  \cE^Y(u,u) = \cE^{\mu,F}(u,u) + \int_{\bR^d} u^2 d\rho^+.
 \end{align}
 In fact, we see that
 the domain of $\cE^Y$ equals the original domain $\cF$ in our setting.
 Moreover, we can prove that $(\cE^{Y}, \cF)$ 
 becomes a transient Dirichlet form.
 Now we consider the following infimum:
 \begin{align}
  \lambda := \lambda(\mu,F) = \inf\left\{\cE^{Y}(u,u)
  : u \in \cF_e,\ \int_{\bR^d}
  u^2 d\rho^+ = 1\right\},
  \la{eq-ground}
 \end{align}
 where $\cF_e$ is the extended Dirichlet space of $(\cE^Y,\cF)$. 
 Let $h$ be a function which attains the infimum in
 the right hand side of \eqref{eq-ground} (if exists).
 We call the function $h$ the ``{\it $\lambda$-ground state}''.
 Note that $h$ satisfies $\cE^{\mu,F}(h,h) = 0$ if $\lambda = 1$. 
 For the existence, we use a method introduced by Takeda (\cite{Ta2}).
 The condition in this method is called {\it ``class (T)''}.
 First we check that our setting satisfies the class (T).
 Using this, we then know that
 $(\cF_e, \cE^{Y})$ is compactly embedded in $L^2(\bR^d, \rho^+)$.
 This compact embedding implies the existence of ground state $h$.
 We find that the function $h$ is unique, strictly positive
 and bounded. 
 Moreover, by the resolvent strong Feller property
 (one of conditions in the class (T)), $h$ is continuous.
 
 Now we define the $\lambda$-harmonicity of functions.
 A function $h$ is said to be $\lambda$-harmonic if
 it satisfies that for any relatively compact open set $U \subset \bR^d$, 
 \begin{align}
  h(x) = \bE_x
  \left[\exp\left(A_{\tau_U} + (\lambda - 1)
  A_{\tau_U}^{\rho^+}\right) h(X_{\tau_U})\right], \quad
  \text{for any $x \in U$}.
 \end{align}
 Let $\lambda$ be the infimum corresponding with the ground state
 \eqref{eq-ground}. 
 After constructing the $\lambda$-ground state $h$, we show that
 the $\lambda$-ground state becomes $\lambda$-harmonic. 
 If $\lambda = 1$, the function $h$ is nothing but
 the $\cH^{\mu,F}$-harmonic function.
 To prove this, we can use the method in \cite{K}
 if $U$ is a very small domain.
 To extend $U$ to any relatively compact open set,
 we use a similar argument in \cite{BB}, but the argument depends on
 the quasi-left-continuity of original process and additive functionals.
 We now consider the non-local perturbation,
 so the associated additive functional
 is discontinuous. Hence we need to take care of this. 

 The structure of this paper is the following:
 In section 2, we collect some fundamental tools for our arguments.
 In particular, the Dirichlet form generated by symmetric stable process,
 Kato class, adjusted decomposition of non-local perturbation and
 related properties. We prove the existence of
 $\lambda$-ground state in section 3.
 In section 4, we show that the $\lambda$-ground state
 is $\lambda$-harmonic function. 
 
 We use $C, C_1, C_2, \ldots$ as positive constants which may
 be different at different occurrences.
 We write $f \in \cB(\bR^d)$ (resp. $f \in \cB^+(\bR^d)$,
 $f \in \cB_b(\bR^d)$)
 if the function $f$ is (resp. non-negative, bounded) Borel measurable
 on $\bR^d$. The notations $dx,dy,\ldots$ mean Lebesgue measures
 on $\bR^d$. 
 
\section{Preliminaries}
Let $\bfM = (\Omega, \cF, \cF_t, X_t, \bP_x, x \in \bR^d)$
be the symmetric (relativistic) $\alpha$-stable process on $\bR^d$.
The generator $\cH$ of symmetric (relativistic) 
$\alpha$-stable process is
\begin{align}
 \cH = m - (-\Delta + m^{2/\alpha})^{\alpha/2}
\end{align}
where $\Delta$ is the Laplacian on $\bR^d$, 
$m \ge 0$ and $0 < \alpha < 2$. 
If $m = 0$ (resp. $m > 0$), then $\bfM$ is called the 
{\it standard symmetric $\alpha$-stable process} (resp.
{\it relativistic symmetric $\alpha$-stable process}).
$\bfM$ is recurrent if $\alpha \ge d = 1$ (resp. $d = 1,2$).
In this paper, we only consider recurrent $\alpha$-stable processes. 
It is known by \cite{BG} and \cite{CS} that their Dirichlet form
on $L^2(\bR^d)$ are 
\begin{eqnarray*}
 & &  \cE^{}(u,v) = C(d, - \alpha) \iint_{\bR^d \times \bR^d \setminus \mathsf{diag}}
 \frac{ (u(x) - u(y))(v(x) - v(y))}{|x-y|^{d + \alpha}}
 \psi(m^{1/\alpha} |x-y|) dx dy, \\
 & & \cF = \cD(\cE^{\alpha}) = \{u \in L^2(\bR^d) :
  \cE^{}(u,u) < \infty\},
\end{eqnarray*}
where
$C(d,-\alpha) = \frac{\alpha \Gamma(\frac{d + \alpha}{2})}{2^{1-\alpha} \pi^{d/2} \Gamma( 1 - \frac{\alpha}{2})}$ and 
\[
 \psi(r) = \frac{I(r)}{I(0)} \quad \textrm{with}\quad  I(r) = \int_{0}^{\infty}
 s^{\frac{d + \alpha}{2} - 1} e^{- \frac{s}{4} - \frac{r^2}{s}} ds.
 \]
 Let $p_t(x,y)\ (t > 0, x,y \in \bR^d)$ be the transition density function and
 $r_{\alpha}(x,y)\ (\alpha > 0)$ the resolvent kernel of $\bfM$, that is,
 \begin{align}
  r_{\alpha}(x,y) = \int_{0}^{\infty} e^{-\alpha t} p_t(x,y) dt.
 \end{align}
 Let $\{p_t, t \ge 0\}$ be the transition semigroup of $\bfM$.
 A positive Borel measure $\nu$ on $\bR^d$ is called {\it ``smooth''} if
 $\mu$ charges no exceptional set and there exists an increasing sequence
 $\{F_{\ell}\}_{\ell=1}^{\infty}$ of compact sets such that
 $\mu(F_{\ell}) < \infty$ for any $\ell$ and $\capa(K \setminus F_{\ell}) \to 0$
 as $\ell \to \infty$ for any compact set $K$, where $\capa$ means
 the $1$-capacity associated with $(\cE,\cF)$.
 We denote by $S$ the class of all smooth measures.
 There is one-to-one correspondence between a smooth measure $\mu$ and
 a positive continuous additive functional $A_t^{\nu}$, it is called
 {\it ``Revuz correspondence''}: for any $f \in \cB^+(\bR^d)$ and
 $\gamma$-excessive function $g$,
 \begin{equation}
  \int_{\bR^d} f(x) g(x) \nu(dx) = \lim_{t \to 0} \frac{1}{t}
   \int_{\bR^d} \bE_x \left[ \int_{0}^{t} f(X_s) dA_s^{\nu}\right] g(x) dx.
   \la{Rev-1}
 \end{equation}

 We define the function $G_K(x,y)$ by
 \begin{equation}
  R_K(x,y) =
  \begin{cases}
   |x-y|^{\alpha - d} & \mathrm{for}\ d > \alpha, \\
   \log(|x-y|^{-1}) & \mathrm{for}\ d = \alpha = 1, \\
   |x-y|^{\alpha - 1} & \mathrm{for}\ \alpha > d = 1.
  \end{cases}
 \la{eq:Kato}
 \end{equation}

 \begin{definition} \rm
 \la{def-Kato-Green}
 For a positive smooth measure $\nu$ on $\bR^d$, $\nu$ is said to be of
 {\it the Kato class} (in notation $\nu \in \bfK$) if
 \begin{equation}
  \lim_{r \to 0} \sup_{x \in \bR^d} \int_{|x-y| < r} |R_K(x,y)| \nu(dy)
   = 0. 
   \la{D:Kato}
 \end{equation}
 \end{definition}

 Next we introduce properties of Kato measure $\bfK$. 
 Put $R_{\alpha} \nu(x) = \int_{\bR^d} r_{\alpha}(x,y) \nu(dy)$
 for $\nu \in \bfK$.

 \begin{theorem}
 [{\rm \cite[Example 5.1]{KT}}] \la{thm-Kato-class}
  If $\mu$ is in $\bfK$, then $\mu$ is smooth.
  Moreover, 
  the following statements are equivalent each other:
  \begin{enumerate}[{\rm (1)}]
   \item $\nu \in \bfK$,
   \item $\ds \lim_{\alpha \to 0} \sup_{x \in \bR^d} R_{\alpha}\nu(x) = 0$, 
   \item $\ds \lim_{t \to 0} \sup_{x \in \bR^d} \bE_x [ A_t^{\nu}] = 0$, where
	 $A_t^{\nu}$ is the positive continuous additive functional
	 corresponding to $\nu$ in the sense of Revuz. 
  \end{enumerate}
 \end{theorem}

 \begin{remark}
 \la{R:Kato}
 For $\alpha > d = 1$, it is known that
 \[
  \nu \in \bfK \quad \lon \quad \sup_{x \in \bR^1} \int_{|x-y|
 \le 1} \nu(dy) < \infty, 
 \]
 by using the same argument in the Remark below \cite[Proposition 3.1]{CZ}.
\end{remark}
 
 The next lemma is a Poincar\'e-type inequality by Stollmann-Voigt \cite{SV}.
 \begin{lemma}
  [\cite{SV}]\la{lem-Poincare-SV}
  Let $\nu \in \bfK$. Then, for $\alpha > 0$
  \begin{equation}
   \int_{\bR^d} u^2(x) \eta(dx) \le \| R_{\alpha} \eta\|_{\infty}
    \cdot \cE_{\alpha}(u,u), \quad u \in \cD(\cE), 
    \la{SV}
  \end{equation}
  where $\cE_{\alpha}(u,u) = \cE(u,u) + \alpha \int_{\bR^d} u^2 dx$. 
 \end{lemma}
 For $\nu \in \bfK$, we see that $\cF$ is included in $L^2(\mu)$. 
Let $(N,H)$ be a L\'evy system for $\bfM$, that is,
\begin{align}
 N(x,dy) &= 2C(d,-\alpha) \frac{\psi(m^{1/\alpha}|x-y|)}{|x-y|^{d+\alpha}}dy,
 \quad H_t = t, \quad t > 0.
\end{align}
It is known that the jump measure $J(dxdy)$
of the Dirichlet form equals $(1/2)N(x,dy)\mu_H(dx)$ in the
general theory of Hunt processes.  
In our setting, note that the Revuz measure
associated with $H_t$ is $\mu_H(dx) = dx$ since $H_t = t$.
Then for any non-negative Borel function $F(x,y)$ on $E \times E$
vanishing on the diagonal set $d = \{(x,x) : x \in \bR^d\}$, 
\begin{align}
 \bE_x\left[\sum_{s \le t} F(X_{s-},X_s)\right]
 = \bE_x\left[\int_{0}^{t} \int_{\bR^d}
 F(X_s,y)N(X_s,dy) ds\right].
 \la{Levy-sys}
\end{align}
Put 
\begin{align}
 NF(x) = \int_{\bR^d} F(x,y) N(x,dy) \quad
 \text{and} \quad \xi_{F}(dx) := NF(x) \cdot dx.
\end{align}
Remark that $\xi_F$ is the Revuz measure of the dual predictable projection
of purely discontinuous additive functional
$A_t := \sum_{s \le t}F(X_{s-},X_s)$.

\begin{definition}
 \la{def-nlp}
 Let $F$ be a symmetric non-negative bounded measurable function on
 $\bR^d \times \bR^d$ with $F(x,x) = 0$ for any $x \in \bR^d$. 
 The function $F$ is said to be in the class $\bfJ$
 if the measure $\xi_{F}$ belongs to $\bfK$.
\end{definition}

For a bounded measurable function $F(x,y)$ on $\bR^d \times \bR^d$
vanishing on the diagonal set,
we put $F^+ = (F \vee 0)$ and $F^{-} = -(F \wedge 0)$.
Then $F = F^+ - F^-$ and $|F| = F^+ + F^-$. 
Following \cite{Li}, set
\begin{align}
 G^{-} := 1 - e^{-F^{-}}, \quad
 G^{+} := \left(e^{F^+} - 1\right) e^{-F^{-}},
\end{align}
and write $G := G^+ - G^-$. Then we easily see $G(x,y) = e^{F(x,y)} - 1$. 
On account of the boundedness of $F$, it is easy to know
there exists a constant
$C > 0$ such that for any $x,y \in \bR^d \times \bR^d$, 
\begin{align}
 C^{-1} F^{+}(x,y) \le G^+(x,y) \le CF^{+}(x,y),
 \quad C^{-1} F^{-}(x,y) \le G^{-}(x,y)
 \le CF^{-}(x,y).
\end{align}
Let $\mu = \mu^+ - \mu^-$ be a signed smooth measure. 
In the sequel, we assume that
$|\mu| = \mu^+ + \mu^- \in \bfK$
and $|F| = F^+ + F^- \in \bfJ$. 
 Let us consider the following non-local Schr\"odinger type operator
 with non-local perturbation:
 \begin{align}
  \cH^{\mu,F} f
  = \cH f  + \mu f + \int_{\bR^d} (e^{F(x,y)} - 1)f(y) N(x,dy) dx,
 \end{align}
 where $\cH$ is the Markov generator of $\bfM$. 
 Let $A_t^F$ be the pure jump additive functional generated
 by $F$, that is, 
 \begin{align}
  A_t^F = \sum_{s \le t} F(X_{s-}, X_s).
 \end{align}
 Then the Feynman-Kac semigroup associated with $\bfM$, $\mu$ and $F$
 is represented as follows:
 \begin{align}
  p_t^{\mu,F}f(x) = \bE[\exp(A_t^{\mu} + A_t^F)f(X_t)]. 
 \end{align}
 Song in \cite{Song} showed the following theorem in the case of
 symmetric $\alpha$-stable processes on $\bR^d\ (d \ge 2)$.
 But the argument is applicable for general symmetric Hunt processes with
 the strong Feller property.
 In the sequel, we occasionally use the notation (SF) as
 the strong Feller property of transition semigroups.
 Since the semigroup $\{p_t, t \ge 0\}$ of $\bfM$ satisfies (SF),
 we obtain the following theorem. 

 \begin{theorem}[{cf. \cite[Theorem 2.6]{Song}}]
  \la{thm-Song}
  Suppose that $|\mu| \in \bfK$ and $|F| \in \bfJ$.
  Then the Feynman-Kac semigroup $\{p_t^{\mu,F}\}$ has
  the strong Feller property, namely, 
  for any $t > 0$, $p_{t}^{\mu,F} (\cB_b) \subset C_b$.
 \end{theorem}

 This semigroup corresponds to the formal
 Schr\"odinger type operator
 \begin{align}
  \cH^{\mu,F}f(x) := \cH f(x) + \mu f(x) +
  \int_{\bR^d}\left(e^{F(x,y)} - 1\right) f(y) N(x,dy) \quad
  \text{for}\ f \in \cB_b(\bR^d). 
 \end{align}
  Hence
 the associated quadratic form is given by 
 \begin{align}
  \cE^{\mu,F}(u,u) &= (-\cH^{\mu,F}u,u)_2 \\
  &= \cE^{}(u,u) - \iint_{\bR^d \times \bR^d}
  u(x)u(y) (e^{F(x,y)} - 1) N(x,dy) dx - \int_{\bR^d} u^2 d\mu,
  \la{p-form}
 \end{align}
 where $(f,g)_2$ means the canonical $L^2$-inner product
 $\int_{\bR^d} f(x) g(x) dx$.
 Noting that $G(x,y) = e^{F(x,y)} - 1$, we see
 \begin{align}
  \cE^{\mu,F} (u,u) &= \cE^{}(u,u) + \int_{\bR^d} u(x)^2
  \mu^-(dx) + \iint_{\bR^d \times \bR^d}
  u(x) u(y) G^-(x,y) N(x,dy) dx \\
  &\hspace{1cm} - \left(\int_{\bR^d} u(x)^2 \mu^+(dx)
  + \iint_{\bR^d \times \bR^d} u(x) u(y) G^+(x,y) N(x,dy) dx\right).
 \end{align}
 We put $\rho^{\pm} := \mu^{\pm} + \xi_{G^{\pm}}$, respectively.
 Since 
 \begin{align}
  &\hspace{-0.5cm} \iint_{\bR^d} u(x) u(y) G^{\pm}(x,y) N(x,dy) dx \\
  &= \int_{\bR^d} u(x)^2 \xi_{G^{\pm}}(dx)
  - \frac{1}{2} \iint_{\bR^d \times \bR^d} (u(x) - u(y))^2 G^{\pm}(x,y)
  N(x,dy) dx
 \end{align}
 by the Fubini's theorem and symmetry of $G^{\pm}(x,y)N(x,dy)dx$, 
 we know that
  \begin{align}
   \begin{split}
  & \hspace{-0.5cm} \cE^{\mu,F}(u,u) \\
  &= \cE^{}(u,u) + \int_{\bR^d} u(x)^2 \rho^-(dx)
  - \frac{1}{2} \iint_{\bR^d \times \bR^d} (u(x) - u(y))^2 G^{-}(x,y)
  N(x,dy) dx \\
  &\hspace{1cm} + \frac{1}{2}\iint_{\bR^d \times \bR^d}
    (u(x) - u(y))^2 G^+(x,y) N(x,dy) dx - \int_{\bR^d} u(x)^2 \rho^+(dx).
   \end{split}
   \la{expr}
  \end{align}
  Note that $\rho^{\pm} \in \bfK$ because $\mu^{\pm}$ and $\xi_{G^{\pm}}$
  belong to $\bfK$.
 
 Now, put $A_t^{\pm} = A_t^{\mu^{\pm}} + A_t^{F^{\pm}}$, respectively. 
 Let $\bfM^{-} = (\bP_x^-,X_t)$ be the subprocess 
 killed by $A^{-}$.
 Then the semigroup of $\bfM^-$ becomes
 \begin{align}
  p_t^- f(x) := \bE_x\left[\exp(-A_t^-) f(X_t)\right] =
  \bE_x\left[\exp(-(A_t^{\mu^-} + A_t^{F^-}))
  f(X_t)\right] \quad \text{for } f \in \cB^+,
 \end{align}
 The associated Dirichlet form is expressed by 
  \begin{align}
 \cE^{-}(u,u) &:= \cE^{-\mu^-,-F^-}(u,u) \\
 &= \cE^{}(u,u)
 + \int_{\bR^d} u(x)^2 \rho^-(dx) - \frac{1}{2}\iint_{\bR^d \times \bR^d}
 (u(x) - u(y))^2 G^{-}(x,y) N(x,dy) dx \\
 &\left(= \iint_{\bR^d \times \bR^d} (u(x) - u(y))^2 e^{-F^-(x,y)} N(x,dy) dx
 + \int_{\bR^d} u(x)^2 \rho^-(dx)\right)
  \end{align}
  for $u \in \cF$.
  By the boundedness of $F^{-}$ and Lemma \ref{lem-Poincare-SV},
  it is easy to see that
  there exists a constant $C > 0$ such that 
  \begin{align}
   C^{-1} \cE_1(u,u) \le \cE_1^{-}(u,u) \le C \cE_1(u,u)
   \quad \text{for any }u \in \cF,
  \end{align}
  on account of $\rho^- \in \bfK$. Hence the domain of $\cE^-$ equals $\cF$ and
  $(\cE^-,\cF)$ is a regular Dirichlet form by the regularity of $(\cE,\cF)$.
   It is clear that the following lemma holds.
   \begin{lemma}
    \la{lem-kproc}
    The killed process $\bfM^-$ is a symmetric and transient
    Hunt process with (SF). 
   \end{lemma}
   By (SF), there exists the density of Green kernel 
   of $\bfM^-$, that is, the Green function.
   Let $R^-(x,y)$ denotes the Green function ($0$-resolvent kernel) of $\bfM^-$,
 that is,
 \begin{align}
  \bE_x\left[\int_{0}^{\infty} e^{-(A_t^{\mu^-} + A_t^{F^-})} f(X_t) dt\right]
  = \int_{\bR^d} R^-(x,y) f(y)dy \quad \text{for } f \in \cB^+(\bR^d).
 \end{align}
 The L\'evy system $(N^{-}, H^{-})$ of $\bfM^{-}$ becomes 
 \begin{align}
  N^{-}(x,dy) = (1 - G^{-}(x,y)) N(x,dy) = e^{-F^{-}(x,y)} N(x,dy), \quad
  H_t^- = H_t = t, \quad t > 0.
 \end{align}

 Let us define the classes of measure and functions with a Green-tightness property.
 We define the $0$-potential of a measure $\eta$ with respect to $\bfM^-$ as follows:
 \begin{align}
  R^{-} \eta := \int_{\bR^d} R^-(x,y) \eta(dy).
 \end{align}
 \begin{definition}
  \la{def-tight}
  \begin{enumerate}[{\rm (1)}]
   \item A smooth measure $\nu$ on $\bR^d$ is in $\bfK_{\infty}^-$ if
	 it is in $\bfK$ and satisfies
	 \begin{align}
	  \lim_{a \to \infty} \sup_{x \in \bR^d} \int_{|y| \ge a} R^-(x,y) \nu(dy) = 0.
	 \end{align}
   \item A non-negative measurable function $F(x,y)$ on $\bR^d \times \bR^d$ is in
	 $\bfJ_{\infty}^-$ if $\xi_{F} \in \bfK_{\infty}^-$. 
  \end{enumerate}
 \end{definition}

 In the sequel, we assume that $\mu^+ \in \bfK_{\infty}^-$
 and $F^+ \in \bfJ_{\infty}^-$.
 We summarize our assumptions:
  \begin{align}
    \begin{cases}
     \text{(killing part)}\ \mu^- \in \bfK,\ F^- \in \bfJ, \\
     \text{(creation part)}\ \mu^+ \in \bfK_{\infty}^-\
     (\subset \bfK),\ 
    F^+ \in \bfJ_{\infty}^-\ (\subset \bfJ). 
    \end{cases}
   \la{ass-01}
  \end{align}

\begin{lemma}
 \la{lem-Green-bdd}
 Suppose that $\eta$ is in $\bfK_{\infty}^-$.
 Then $R^{-} \eta$ is bounded. 
\end{lemma}
\begin{proof}
Note that if $\eta \in \bfK_{\infty}^-$,
\begin{align}
 \lim_{t \to 0} \sup_{x \in \bR^d} \bE_x^{-}[A_t^{\eta}]
 \le \lim_{t \to 0} \sup_{x \in \bR^d} \bE_x [A_t^{\eta}] = 0.
\end{align}
 Thus $\eta$ is the Kato class with respect to $\bfM^-$. 
By \cite[Lemma 4.1 (2)]{KK0}, $\bfK_{\infty}^-$ is identical to the class
introduced in \cite[Definition 2.2 (1)]{Ch}.
 Using \cite[Proposition 2.2]{Ch}, we obtain this lemma.
\end{proof}

Since $F^+ \in \bfK_{\infty}^-$,
\begin{align}
 &\hspace{-0.5cm} \sup_{x \in \bR^d} \int_{|y| \ge a} R^{-}(x,y) \xi_{G^+}(dy) \\
 &= \sup_{x \in \bR^d} \int_{|y| \ge a} R^{-}(x,y) NG^+(y) dy
 \le C \sup_{x \in \bR^d} \int_{|y| \ge a} R^{-}(x,y) NF^+(y)dy \\
 &= C \sup_{x \in \bR^d} \int_{|y| \ge a} R^{-}(x,y) \xi_{F^+}(dy)
 \to 0 \quad \text{as }a \to \infty.
\end{align}
Hence, $\xi_{G^+} \in \bfK_{\infty}^-$, so does $\rho^+$. 
 On account of Lemma \ref{lem-Green-bdd}, 
 $\rho^+$ is $R^{-}$-bounded, that is,
 \begin{equation}
  \sup_{x \in \bR^d} R^{-}(\eta)(x)
   := \sup_{x \in \bR^d}
   \int_{\bR^d}R^{-}(x,y) \eta(dy) = \sup_{x \in \bR^d} \bE_x^{-}
   [A_{\infty}^{\eta}]  < \infty.
   \la{GB}
 \end{equation}

 Next, we transform $\bfM^-$ by $\exp(A_t^{+})$.
 \begin{align}
  p_t^{\mu,F} f(x) = \bE_x^-[e^{A_t^+} f(X_t)]
  = \bE_x\left[e^{-A_t^-} e^{A_t^+} f(X_t) \right]
 \end{align}

 \begin{proposition}
  \la{prop-semibdd}
  Assume that $\rho^- \in \bfK$ and $\rho^+ \in \bfK_{\infty}^-$.
  Then the quadratic form $(\cE^{\mu,F},\cF)$ is lower semi-bounded.
 \end{proposition}
  \begin{proof}
   We have already known that the domain of $\cE^{-}$ is $\cF$ if $\rho^- \in \bfK$. 
   By the equation \eqref{expr}, for any $u \in \cF$,
    \begin{align}
     \begin{split}
    \cE^{\mu,F}(u,u) &= \cE^-(u,u) + \frac{1}{2}\iint_{\bR^d \times \bR^d} (u(x) - u(y))^2
    G^+(x,y) N(x,dy) dx - \int u^2 d\rho^+ \\
    &= \cE^{-}(u,u) + \frac{1}{2} \iint_{\bR^d \times \bR^d} (u(x) - u(y))^2
      \left(e^{F^+(x,y)} - 1\right) N^{-}(x,dy) dx - \int u^2 d\rho^+
     \end{split}
     \la{expr-2}
    \end{align}
   From $\rho^+ \in \bfK_{\infty}^-$ and Lemma \ref{lem-Poincare-SV} we can deduce
   $\cE_{\alpha}^{\mu,F} \sim \cE_1^-$ for a large $\alpha > 0$. Using the
   fact that $\cE_1^- \sim \cE_1$, this completes the proof.
  \end{proof}

  \section{The existence of ground states}
  Following arguments in \cite{Li} and \cite{Tsu2},
  we show the existence of ground states. 
 In the sequel, we only consider the case that
 $\mu^{\pm} \not\equiv 0$ and $F^{\pm} \not\equiv 0$.
 Let $(\cF_e,\cE^-)$ be the extended Dirichlet space of $(\cE^{-}, \cF)$.

 We define another quadratic form by
 \begin{align}
  \cE^Y(u,u) := \cE^{\mu,F}(u,u) + \int_{\bR^d} u^2 d\rho^+.
  \la{Y-form}
 \end{align}
 Then by the expression \eqref{expr-2}, $(\cE^F, \cF)$ is non-negative definite
 since 
 \begin{align}
  \cE^Y(u,u) &:= \cE^{-}(u,u) + \frac{1}{2} \iint_{\bR^d \times \bR^d}
  (u(x) - u(y))^2 G^{+}(x,y) N(x,dy) dx.
 \end{align}
 Moreover, by the boundedness of $G^+$, it follows that 
 $\cE^Y \sim \cE^-$, hence $(\cE^Y,\cF^-)$ is a regular
 Dirichlet form on $L^2(\bR^d)$.
 Hence we can get a transient Hunt process associated with $(\cE^Y,\cF)$. 
 We denote it by $\bfY$.

 \begin{lemma}
  \la{lem-Y}
  The process $\bfY$ is a symmetric
  and transient Hunt process with (SF).
 \end{lemma}
 \begin{proof}
  As mentioned above,
  it is clear that $\bfY$ is a transient and symmetric Hunt process.
  Since $|\mu| + \rho^+ \in \bfK$ and $|F| \in \bfJ$,
  we find (SF) by Theorem \ref{thm-Song}.
 \end{proof}

 Let $R^Y(x,y)$ be the Green function of $\bfY$.
 Note that the extended Dirichlet space is $(\cF_e^-,\cE^Y)$. 
 
 In this section we consider the following infimum:
 \begin{equation}
  \lambda := \lambda(\mu,F)
   = \inf \left\{\cE^{Y}(u,u):
	   u \in \cF_e,\
	   \int_{\bR^d} u^2 d\rho^+ = 1\right\}.
  \la{e:GS}
 \end{equation}

 \begin{definition}
  \la{D:GS}
  Let $h$ be the function attaining the infimum of \eqref{e:GS}
  (if it exists). We call $h$ {\it the $\lambda$-ground state
  of $\cH^{\mu,F}$}.
 \end{definition}

 To study the existence of the $\lambda$-ground state of $\cH^{\mu,F}$,
 it is crucial to check conditions of {\it class (T)} introduced in Takeda \cite{Ta2}. 
 We then see from Lemma \ref{lem-Poincare-SV} and \eqref{GB}
 that for any $u \in \cF_e$, 
 \begin{align}
  \int_{\bR^d} u^2 d\rho^+
  &\le \|R^- \rho^+\|_{\infty} \cE^{-}(u,u) \le \|R^- \rho^+\|_{\infty}
  \cE^{Y}(u,u), 
 \end{align}
 hence we can deduce
 \begin{align}
  \lambda \ge \frac{1}{\|R^{-} \rho^+\|_{\infty}} > 0.
 \end{align}
  If $\lambda = 1$ and
 there exists a function $h$ such that $h$ attains the infimum in \eqref{e:GS},
 this function $h$ satisfies
 \begin{align}
  (-\cH^{\mu,F}h,h)_2 = \cE^{\mu,F}(h,h) = 0.
 \end{align}
 We can regard that the function $h$ is harmonic with respect to
 the Schr\"odinger type operator $\cH^{\mu,F}$ with
 the non-local perturbation.

 To construct the $\lambda$-ground state, we must impose an additional
 assumption.

 \begin{definition}
  \la{def-Ainf}
  Suppose that $F(x,y)$ is a non-negative
  symmetric bounded measurable function
  on $\bR^d \times \bR^d$ vanishing on the diagonal.
  The function $F$ is said to be in the class $\bfA_{\infty}^-$
  if for any $\varepsilon > 0$ there is a Borel subset $K = K(\varepsilon)$
  of finite $\xi_G$-measure and a constant $\delta = \delta(\varepsilon) > 0$
  such that
  \begin{align}
   \sup_{(x,w) \in (E \times E) \setminus d}
   \iint_{(K \times K)^c} R^{-}(x,y)
   \frac{F(y,z) R^{-}(z,w)}{R^{-}(x,w)}N(y,dz) dy \le \varepsilon
  \end{align}
  and for all measurable sets $B \subset K$ with
  $\xi_F(B) < \varepsilon$,
  \begin{align}
   \sup_{(x,w) \in (E \times E) \setminus d}
   \iint_{(B \times E) \cup (E \times B)} R^{-}(x,y)
   \frac{F(y,z) R^{-}(z,w)}{R^{-}(x,w)}N(y,dz) dy \le \varepsilon.
  \end{align}
 \end{definition}

 \begin{lemma}
  \la{lem-tight-02}
  Assume that $F^{+} \in \bfA_{\infty}^-$. 
  If $\nu \in \bfK_{\infty}^-$, then $\nu$ is a Green-tight
  Kato measure associated with $\bfY$, that is,
  \begin{align}
   \lim_{a \to \infty} \sup_{x \in \bR^d} \int_{|y| \ge a}
   R^{Y}(x,y) \nu(dy) = 0.
  \end{align}
 \end{lemma}
 \begin{proof}
  Suppose that $\eta$ is in $\bfK_{\infty}^-$.
  Let $\lambda$ be a constant in \eqref{e:GS}.
  Then we can deduce that for any $u \in \cF_e^-$,
  \begin{align}
   \cE^{Y}(u,u) \ge \lambda \int_{\bR^d} u^2 d\rho^+
   \ge \lambda \int_{E} u^2 d\xi_{G^+}.
  \end{align}
  This implies
  \begin{align}
   \inf\left\{\cE^Y(u,u) : u \in \cF_e^-,\
   \lambda \int_{\bR^d} u^2 d\xi_{G^+} = 1\right\}
   \ge 1 > 0,
  \end{align}
  hence by \cite[Theorem 1.1]{KK0},
  \begin{align}
   \sup_{x \in \bR^d} \bE_x^-
   \left[\exp\left(-A_{\zeta}^{\xi_{G^+}}
   + A_{\zeta}^{F^+}\right)\right] < \infty,
   \la{eq-gauge}
  \end{align}
  where $\zeta$ denotes the life time of $\bfM^-$.
  This means the gaugeability of $-\xi_{G^+} + F^+$
  with respect to $\bfM^-$. 
  By \cite[Lemma 3.9 (2), Theorem 3.10]{Ch},
  since the gaugeability is equivalent to the conditional gaugeability
  in this setting, 
  there exists a constant $K > 0$ such that
  \begin{align}
   R^Y(x,y) \le KR^{-}(x,y)
  \end{align}
  for any $x,y \in \bR^d$.
  The proof is completed.
 \end{proof}

 \begin{lemma}
  \la{lem-irr}
  The process $\bfY$ is irreducible. 
 \end{lemma}
 \begin{proof}
  By the gaugeability in \eqref{eq-gauge}, we see 
  \begin{align}
   M_t = \exp\left(-A_t^- - A_{t}^{\xi_{G^+}} + A_t^{F^+}\right) > 0
   \quad \text{for}\ t < \zeta. 
  \end{align}
  Since the semigroup of $\bfY$ is represented by
  \begin{align}
   p_t^{\bfY}f(x) = \bE_x\left[M_t f(X_t)\right], 
  \end{align}
  The irreducibility of $\bfM$ implies the irreducibility of $\bfY$. 
  
 \end{proof}

 \begin{theorem}
  \la{thm-main-01}
  Suppose that $\mu$ and $F$ satisfy \eqref{ass-01}.
  In addition, assume that $F^+ \in \bfA_{\infty}^-$.
  Then there exists a positive continuous
  bounded $\lambda$-ground state $h$ in \eqref{e:GS}.
 \end{theorem}
  \begin{proof}
   By Lemma \ref{lem-tight-02},
   note that the measure $\rho^+$ belongs to $\bfK_{\infty}(\bfY)$
   since $\rho^+$ is in $\bfK_{\infty}^-$.
   Combining Lemma \ref{lem-irr} and \cite[Theorem 4.8]{Ta}, 
   we see from  the existence of
   $\lambda$-ground state $h$. Moreover, we can take the function
   $h$ which is strictly positive, bounded and continuous on $\bR^d$. 
  \end{proof}

  \section{Harmonicity of ground states}
  In this section, we prove that the ground state $h$ obtained from
  Theorem \ref{thm-main-01} satisfies 
  $\cH^{\mu + (1-\lambda) \rho^+, F} h = 0$,  
  in particular, $h$ is a $\cH^{\mu,F}$-harmonic function if $\lambda = 1$.
  We impose the following assumption to prove that
  the ground state becomes the harmonic function with respect to
  $\cH^{\mu + (1 - \lambda)\rho^+ + F}$:
  \begin{itemize}
   \item[{\bf (A)}] for any $z \in \bR^d$, there exists a constant $r > 0$
		such that $\|R^{Y}(1_{B(z,r)} \mu)\|_{\infty} < 1/\lambda$
  \end{itemize}

  We now define the harmonic function probabilistically.
  For a open set $D \subset \bR^d$, $\tau_D$ denotes the
  exit time of $D$, $\tau_D := \inf\{t > 0 : X_t \not\in D\}$.
  We consider a connected non-empty open set $D$ of $\bR^d$.
  The boundary point $z \in \partial D$ is
  said to be {\it regular} if $\bP_z (\tau_D = 0) = 1$.
  Denote by $(\partial D)_r$ the set of regular points in boundary.
  $D$ is said to be {\it regular} if $(\partial D)_r = \partial D$.
  The part process $\bfM^{D}$ of $\bfM$
  on $D$ is defined as the process
  killed upon leaving $D$. Note that $\bfM^{D}$ is transient.
  Let $R^D(x,y)$ be a Green function of $\bfM^D$ and
  $(\cE^{}, \cF^D)$ the part Dirichlet
  form of $(\cE^{}, \cF^D)$ on
  $D$ (see \cite[Section 4.4]{FOT}).

  \begin{definition}
   \la{def-4-3}
   Let $D$ be an open set in $\bR^d\ (d = 1,2)$.
   A bounded continuous function $h$ on $D$
    is said to be {\it $\cH^{\mu,F}$-harmonic on $D$},
   if for any relatively compact open set $U$ of $D$,
   \begin{equation}
    h(x) = \bE_x [ \exp (A_{\tau_U}) h(X_{\tau_U})],
     \quad \textit{for any}\ x \in U.
   \end{equation}
  \end{definition}

  For any domain $D$, set
  \begin{align}
   u_{D,A}(x) = \bE_x^-[\exp\left(A_{\tau_D}^+\right)].
  \end{align}
  If $\sup_{x \in \bR^d} u_{D,A}(x)$ is finite, then
  we say that the pair $(D,A)$ is gaugeable with respect to
  $\bfM^-$.

  \begin{lemma}
   \la{lem-gauge-01}
   Let $D$ be a bounded regular domain with 
   $\rho^+(D) > 0$. Then $(D,A^+ + (\lambda - 1)A^{\rho^+})$ is
   gaugeable with respect to $\bfM^-$.
  \end{lemma}
  \begin{proof}
   Since $\rho^+ \in \bfK_{\infty}^-$,
   we know that for any $\varepsilon > 0$ there exists a
   Borel set $K = K(\varepsilon)$ such that for all Borel
   subset $K = K(\varepsilon)$ of $D$ with $\rho^+(K) < \infty$
   and a constant $\delta > 0$ such that for all $\rho^+$-measurable
   set $B \subset K$ with $\rho^+(B) < \delta$,
   \begin{align}
    \sup_{x \in D} R^D(1_{B \cup K^c} \rho^+)(x)
    = \sup_{x \in D} \int_{B \cup K^c} R^D(x,y) \rho^+(dy)
    < \varepsilon.
   \end{align}
   
   Now we claim that
   \begin{align}
    \inf\left\{\cE^{Y}(u,u) - \lambda \int_D u^2 d\rho^+ :
    u \in \cF^D,\ \int_{D} u^2 dx = 1\right\} > 0.
    \la{eq-1-eigen}
   \end{align}
   If the left-hand side equals zero,
   there exists a function $u_0 \in \cF^D$ such that
   \begin{align}
    \cE^Y(u_0,u_0) = \lambda \int_{D}u_0^2 d\rho^+
   \end{align}
   since $(\cF^D,\cE_1^{})$ is
   compactly embedded in $L^2(D,dx)$ by \cite[Theorem 4.1]{Ta2}.
   Noting $\int_{D} u_0^2 d\rho^+ > 0$ because $u_0(x) > 0$ q.e. on $D$ and
   $\rho^+(D) > 0$, the function $u_0/\sqrt{\int_{D} u_0^2 d\mu^-}$
   attains the infimum \eqref{e:GS}. Hence
   $u_0(x) > 0$ for q.e. $x \in \bR^d$, this is contradictory that
   $u_0 \in \cF^D$.
   
   In view of \cite[Remark 2.3]{TT},
   the inequality \eqref{eq-1-eigen} implies
   \begin{align}
    \inf\left\{\cE^Y(u,u) : u \in \cF^D,\ \lambda \int_{D} u^2 d\rho^+ = 1
    \right\} > 1.
    \la{eq-2-eigen}
   \end{align}
   Hence we prove that $(D, \lambda\rho^+)$ is
   gaugeable with respect to $\bfY$ by \eqref{eq-2-eigen}
   and \cite[Theorem 5.1]{Ch}, that is,
   \begin{align}
    & \sup_{x \in \bR^d} \bE_x^Y\left[\exp(\lambda A_{\tau_D}^{\rho^+})\right]
    < \infty \\
    &\iff \sup_{x \in \bR^d} \bE_x
    \left[\exp\left(-A_{\tau_D}^- + \sum_{s \le \tau_D} F^+(X_{s-},X_s)
    - A_{\tau_D}^{\xi_{G^+}} + \lambda A_{\tau_D}^{\rho^+}\right)\right] < \infty \\
    &\iff \sup_{x \in \bR^d} \bE_x^{-} \left[\exp\left(\sum_{s \le \tau_D} F^{+}(X_{s-},X_s)
    + \lambda A_{\tau_D}^{\rho^+} - A_t^{\xi_{G^+}} \right)\right] < \infty \\
    &\iff \sup_{x \in \bR^d} \bE_x^{-}\left[\exp\left(A_{\tau_D}^+ +
    (\lambda - 1) A_{\tau_D}^{\rho^+}\right)\right] < \infty
   \end{align}
   Therefore the proof is completed. 
  \end{proof}

  Note that if $\lambda = 1$, Lemma \ref{lem-gauge-01} implies 
  \begin{align}
   \sup_{x \in D} \bE_x\left[\exp\left(A_{\tau_D}\right)\right] < \infty.
  \end{align}

  In the sequel, we put $A^{\eta}_t = A_t^{+} + (\lambda - 1)A_t^{\rho^+}$.
  
  \begin{lemma}
   \la{lem-gaugeability}
   Suppose that $D$ is a bounded regular domain such that
   $\rho^+(D) > 0$. 
   For any bounded domain $B$ of $D$, $(B, A^{\eta})$ is gaugeable with respect
   to $\bfM^{-}$. 
  \end{lemma}
  \begin{proof}
  Note that $\tau_B \le \tau_D$, $\bP_x^{-}$-a.s. Hence
   by the strong Markov property,
   \begin{align}
    \bE_x^-\left[\exp\left(A_{\tau_D}^{\eta}\right)\right]
    &= \bE_x^{-}\left[\tau_B < \tau_D,\ \exp\left(A_{\tau_D}^{\eta}\right)\right]
    + \bE_x^{-}\left[\tau_B = \tau_D,\ \exp\left(A_{\tau_D}^{\eta}\right)\right] \\
    &= \bE_x^{-} \left[\tau_B < \tau_D,\
    \exp\left(A_{\tau_B}^{\eta}\right)
    \bE_{X_{\tau_B}}^{-}\left[\exp\left(A_{\tau_D}^{\eta}\right)\right]\right] \\
    &\hspace{1cm} + \bE_x^{-}\left[\tau_L = \tau_B,\ \exp\left(A_{\tau_B}^{\eta}\right)
    \right] \\
    &\ge \left(\inf_{z \in \bR^d} \bE_z^{-}
    \left[\exp\left(A_{\tau_D}^{\eta}\right)\right]\right)
    \bE_x^{-}\left[\tau_B < \tau_D,\ \exp\left(A_{\tau_B}^{\eta}\right)\right] \\
    &\hspace{1cm} + \bE_x^{-}\left[\tau_L = \tau_B,\ \exp\left(A_{\tau_B}^{\eta}\right)
    \right] \\
    &\ge \left(\inf_{z \in \bR^d} \bE_z^{-}\left[\exp\left(A_{\tau_D}^{\eta}\right)\right]
    \wedge 1\right) \bE_x^{-}\left[\exp\left(A_{\tau_B}^{\eta}\right)\right].
   \end{align}
   By the Jensen's inequality and \eqref{GB}, we know that
   \begin{align}
    \inf_{z \in \bR^d}
    \bE_z^{-}[\exp(\lambda A_{\tau_D}^{\eta})]
     &\ge \inf_{z \in \bR^d} \bE_z^{-}
     [\exp(-\lambda A_{\tau_D}^{\eta})] 
    \ge  \exp \left(- \lambda \sup_{z \in \bR^d} \bE_z^{-}
      [A_{\tau_D}^{\eta}]\right)
    > 0.
    \la{gauge-01}
   \end{align}
   Hence if $u_{D, A^{\eta}}$ is bounded on $\bR^d$,
   $u_{B, A^{\eta}}$ is also bounded.
   If we take  $D$ as a bounded regular domain so large that $\rho^+(D) > 0$,
   the function $u_{B, A^{\eta}}$ is bounded on $\bR^d$ since 
   $u_{D,A^{\eta}}$ is bounded on $\bR^d$ by Lemma \ref{lem-gauge-01}.
  \end{proof}

  Note that the $\cH^{\mu + (\lambda-1)\rho^+, F}$-harmonic function on $D$ satisfies
  that for any relatively compact open set $U \subset D$, 
  \begin{align}
   h(x) &= \bE_x\left[\exp(A_{\tau_U} +
   (\lambda - 1) A_{\tau_U}^{\rho^+}) h(X_{\tau_U})\right],
   \quad \text{for any } x \in U. \\
  \end{align}
  We call the function $h$ ``{\it $\lambda$-harmonic}''.
  Note that the $\cH^{\mu,F}$-harmonicity is equivalent to the $1$-harmonicity.

 \begin{lemma}
   \la{lem-Kato-1}
  Let $D_1, D_2$ be relatively compact open subsets in $\bR^d$.
  Suppose that $(D_1 \cup D_2, A^{\eta})$ is gaugeable with respect to $\bfM^-$. 
   If $h$ is a strictly positive $\lambda$-harmonic on
   $D_1$ and $D_2$,
   then $h$ is $\lambda$-harmonic on $D_1 \cup D_2$.
  \end{lemma}
  \begin{proof}
   Following the argument of
   the proof in \cite[Lemma 4.4]{BB},
   we take care of the discontinuity of additive functionals.

   Let $U_1$ (resp. $U_2$) be any relatively compact open set in $D_1$ (resp. $D_2$).
   Let $D = D_1 \cup D_2$ and $U = U_1 \cup U_2$.
   Noting $A^{\eta} = A^+ + (\lambda - 1)A^{\rho^+}$, 
   it suffices to prove that for every relatively compact open set $U$ of $D$,
   \begin{equation}
    h(x) = \bE_x^-[\exp(A_{\tau_U}^{\eta}) h(X_{\tau_U})], \quad
     x \in U.
     \la{eq-4-11}
   \end{equation}
   By the definition of $\lambda$-harmonicity on $D_1$ and $D_2$,
   we know for $i = 1,2$
   \begin{equation}
    h(x) = \bE_x^- [\exp (A_{\tau_{U_i}}^{\eta}) h(X_{\tau_{U_i}})], \quad  x \in U_i.
     \la{eq-4-12}
   \end{equation}
   Let $x \in U$ be fixed. We define
   $T_0 = 0$,
   \[
   T_{2n-1} = T_{2n-2} + \tau_{U_1} \circ \theta_{T_{2n-2}}\ \text{and}\ 
   T_{2n} = T_{2n-1} + \tau_{U_2} \circ \theta_{T_{2n-1}}, \quad
   n = 1,2, \ldots.
   \]
   First we will show that
   \begin{equation}
    h(x) = \bE_x^- [ \exp(A_{T_k}^{\eta}) h(X_{T_k})], \quad k \ge 0.
     \la{eq-4-14}
   \end{equation}
   Clearly, \eqref{eq-4-14} holds for $k = 0$ and $1$.
   Suppose that for some $n \in \bN$ it follows that
   \[
    h(x) = \bE_x [ \exp(A_{T_{2n-1}}^{\eta}) h(X_{T_{2n-1}})].
   \]
   We have $X_{T_{2n-1}} \in U_2 \setminus U_1$ on the set
   $\{ T_{2n - 1} < \tau_U\}$ and $T_{2n - 1} = T_{2n}$ on the set
   $\{ T_{2n - 1} = \tau_U\}$.
   By \eqref{eq-4-12} and the strong Markov property we have
   \begin{align}
    h(x) &= \bE_x^{-}\left[\exp\left(A_{T_{2n-1}}^{\eta}\right)h(X_{T_{2n-1}})\right] \\
    &{
    = \bE_x^{-}\left[ T_{2n - 1} = \tau_U,\ \exp \left(A_{T_{2n-1}}^{{\eta}}\right)
    h(X_{T_{2n-1}})\right]} \\
    &{
    \hspace{1cm} + \bE_x^{-} \left[ T_{2n - 1} < \tau_U,\ \exp\left(A_{T_{2n-1}}^{\eta}\right)
    \bE_{X_{T_{2n-1}}}^{-}\left[\exp\left(A_{\tau_{U_2}}\right) h(X_{\tau_{U_2}})\right]\right]} \\
    &{
    = \bE_x^{-}\left[ T_{2n - 1} = \tau_U,\ \exp \left(A_{T_{2n-1}}^{{\eta}}\right)
    h(X_{T_{2n-1}})\right]} \\
    &{
    \hspace{1cm} + \bE_x^{-} \bigg[ T_{2n - 1} < \tau_U,\ \exp\left(A_{T_{2n-1}}^{\eta}\right)} \\
    &{
    \hspace{2cm}
    \times \bE_x^{-}\left[\exp\left(A_{\tau_{U_2}
    \circ \theta_{T_{2n-1}}}^{\eta} \circ \theta_{T_{2n-1}}\right)
    h(X_{T_{2n-1} + \tau_{U_2} \circ \theta_{T_{2n-1}}})\bigg| \cF_{T_{2n-1}}\right]
    \bigg]} \\
    &{
    = \bE_x^{-} \left[ T_{2n - 1} = \tau_U,\ \exp \left(A_{T_{2n-1}}^{{\eta}}\right)
    h(X_{T_{2n-1}})\right]} \\
    &{
    \hspace{2cm} + \bE_x^{-} \left[T_{2n-1} < \tau_U,\ 
    \exp\left(A_{T_{2n-1} + \tau_{U_2}\circ \theta_{T_{2n-1}}}^{\eta}\right)
    h(X_{T_{2n-1} + \tau_{U_2} \circ \theta_{T_{2n-1}}})\right]} \\
    &= \bE_x^{-} [ T_{2n-1} = \tau_U,\ \exp(A_{T_{2n}}^{{\eta}}) h(X_{T_{2n}})]
    + \bE_x^{-} [ T_{2n-1} < \tau_U , \exp(A_{T_{2n}}^{{\eta}}) h(X_{T_{2n}})] \\
    &= \bE_x^{-} [ \exp(A_{T_{2n}}^{{\eta}}) h(X_{T_{2n}})].
   \end{align}
   Similarly, if for some $n \in \bN$
   \[
    h(x) = \bE_x^{-} [\exp(A_{T_{2n}}^{{\eta}}) h(X_{T_{2n}})],
   \]
   then
   \[
    h(x) = \bE_x^{-} [\exp(A_{T_{2n + 1}}^{{\eta}}) h(X_{T_{2n+1}})].
   \]
   By induction, \eqref{eq-4-14} holds for any $k$.
   Since any Hunt process (including $\bfM$) admits no predictable jump (cf. \cite[Theorem A.3.2]{FOT}),
   we see that $A^{\eta}$ (especially in pure jump part) is quasi-left-continuous.
   Hence, 
   it follows $\bP_x$-a.s. that $T_{k} \uparrow \tau_U,\ A_{T_k}^{\eta} \to A_{\tau_U}^{\eta}$
   and $X_{T_k} \to X_{\tau_U}$, $\bP_x$-a.s. 
   as $k \to \infty$. By the continuity of $h$ on $\ov{U}$, we obtain 
   \begin{equation}
    \exp(A_{T_k}^{{\eta}}) h(X_{T_k}) \to \exp(A_{\tau_U}^{{\eta}}) h(X_{\tau_U}),\
     \text{$\bP_x$-a.s.}
     \ \ttr{as $k \to \infty$}.
     \la{eq-4-15}
   \end{equation}
   By \eqref{eq-4-14} and Fatou's lemma,
   \begin{equation}
    h(x) \ge \bE_x^{-} [\exp(A_{\tau_U}^{{\eta}}) h(X_{\tau_U})].
     \la{eq-4-16}
   \end{equation}
   Since $h > 0$ on $D \sm U$ and $\bP_x(X_{\tau_U} \in D \sm U) > 0$,
   the right-hand side of \eqref{eq-4-16} is strictly positive.
   Since $D$ is bounded, 
   we see $(U, A^{\eta})$ is gaugeable with respect to $\bfM^{-}$ by
   Lemma \ref{lem-gaugeability}. 
   To complete the proof, we note that, for $k \in \bN$ we have
   \[
   \bE_x^{-}[\exp(A_{\tau_U}^{{\eta}}) |\cF_{T_k}] = \exp (A_{T_k}^{{\eta}})
   \bE_{X_{T_k}}^-
   [\exp(A_{\tau_U}^{{\eta}})], \quad \text{$\bP_x$-a.s.}
   \]
   hence $\exp(A_{T_k}^{{\eta}}) \le c \bE_x^{-} [\exp(A_{\tau_U}^{{\eta}}) | \cF_{T_k}]$,
   where $c = [ \inf_{y \in \bR^d} \bE_y^- [\exp(A_{\tau_U}^{{\eta}}) ]]^{-1}$ is
   finite on account of \eqref{gauge-01}.
   Thus $\{\exp(A_{T_k}^{\eta}),\ k \ge 1\}$ is
   uniformly integrable, so is $\{\exp(A_{T_k}^{\eta}) h(X_{T_k}), k \ge 1\}$
   on account of the boundedness of $h$. 
   We put $\cP = \{ T_k < \tau_U,\ k = 1,2, \ldots\}$ and
   $\cO = \{ T_k = \tau_U\ \ttr{for some $k \in \bN$}\}$.
   By the uniform integrability and
   \eqref{eq-4-15} we obtain \\
   \if
   {\ma (14)から$\exp(A_{T_k}^{\eta})h(X_{T_k})$は
   $\exp(A_{\tau_U}^{\eta}) h(X_{\tau_U})$に概収束するので，
   これは確率収束している．\\
   そして，$\{\exp(A_{T_k}) h(X_{T_k}),\ k \ge
   1\}$は一様可積分なので，これは$L^1$-収束も意味している．}
   \fi
   \begin{equation}
    \lim_{k \to \infty} \bE_x^{-} [T_k < \tau_U , \exp(A_{T_k}^{{\eta}})
     h(X_{T_k})]
     = \bE_x^{-} [\cP, \exp(A_{\tau_U}^{{\eta}}) h(X_{\tau_U})].
     \la{eq-4-17}
   \end{equation}
   By the monotone convergence theorem we have
   \begin{equation}
    \lim_{k \to \infty} \bE_x^{-} [ T_k = \tau_U , \exp(A_{T_k}^{{\eta}})
     h(X_{T_k})] = \bE_x^{-} [ \cO , \exp(A_{\tau_U}^{{\eta}}) h(X_{\tau_U})].
     \la{eq-4-18}
   \end{equation}
   By \eqref{eq-4-14}, \eqref{eq-4-17} and \eqref{eq-4-18},
   the equality \eqref{eq-4-11}
   holds. 
  \end{proof}

  \begin{proposition}
   \la{prop-4-2}
   Suppose that the assumption {\rm (A)} holds.
   Let $h$ be the $\lambda$-ground state in Theorem \ref{thm-main-01}.
   Then for any $z \in \bR^d$,
   there exists a constant $r > 0$ such that  the function $h$
   satisfies 
   \begin{equation}
    h(x) = \bE_x \left[ \exp \left(A_{\tau_{B(z,r)}}
		  + (\lambda - 1)A_{\tau_{B(z,r)}}^{\rho^+}\right)
		  h(X_{\tau_{B(z,r)}})\right], \quad x \in B(z,r). 
   \end{equation}
  \end{proposition}
  \begin{proof}
   We see from \cite[Theorem 3.2]{Ta}
   that for any $\varphi \in C_0^{\infty}(B(z,r))$,
   \begin{equation}
    \cE^{Y}(h, \varphi) - \lambda \int_{\bR^d}
     h\varphi\, d\rho^+ = 0.
     \la{4-1}
   \end{equation}
   Put $\cE^{Y_0}(u,u) = \cE^{Y}(u,u) - \lambda \int u^2 1_{B(z,r)} d\rho^+$. 
   Since $1_{B(z,r)} \rho^+$ is in $\bfK_{\infty}(\bfY)$, we find
   that by the Schwarz inequality and Theorem \ref{lem-Poincare-SV}
   \begin{align}
    |\cE^{Y_0}(u,v)| &\le
   (1 + \lambda\| R^{Y} (1_{B(z,r)} \rho^+) \|_{\infty})
   \sqrt{\cE^{Y}(u,u)}\sqrt{\cE^{Y}(v,v)}, \quad
   u,v \in \cF
   \la{eq-K1}
   \end{align}
   Moreover, we see from the assumption (A) that for any
   $z \in \bR^d$, there exists $r > 0$ such that
   $\lambda \| R^Y ( 1_{B(z,r)} \rho^+)\|_{\infty} < 1$.
   Hence
   \begin{equation}
    \cE^{Y_0}(u,u) \ge (1 - \lambda\| R^{Y} (1_{B(z,r)} \rho^+)\|_{\infty})\,
     \cE^{Y}(u,u),
     \quad u \in \cF.
     \la{eq-K2}
   \end{equation}
   Therefore we find that $(\cE^{Y_0}, \cF)$ is
   the strongly sectorial coercive closed form on $L^2(\bR^d; dx)$
   with lower bound $0$. 
   Then we see that the function $h$ satisfies
   \begin{align}
    h(x) &= \bE_x^{Y} \left[ \exp
			  \left( \lambda A_{\tau_{B(z,r)}}^{\rho^+} \right)
			   h(X_{\tau_{B(z,r)}})
					  \right] \\
    &= \bE_x \left[ \exp \left(A_{\tau_{B(z,r)}}
			   - A_{\tau_{B(z,r)}}^{\rho^+} + 
			   \lambda A_{\tau_{B(z,r)}}^{\rho^+}\right)
    h(X_{\tau_{B(z,r)}})\right] \\
    &= \bE_x\left[\exp\left(A_{\tau_{B(z,r)}}
    + (\lambda - 1) A_{\tau_{B(z,r)}}^{\rho^+}\right)
    h(X_{\tau_{B(z,r)}})\right]
    , \quad x \in B(z,r)
   \end{align}
   by \cite[Theorem 7.4]{K} and (SF) of $\bfY$. 
   We obtain this proposition.
  \end{proof}

  \begin{theorem}
   \la{T:HF}
   The ground state $h$ is $\lambda$-harmonic on $\bR^d$, i.e.,
   $h$ satisfies that for any relatively compact open subset $U$ in $\bR^d$,
   \[
   h(x) = \bE_x\left[\exp\left(A_{\tau_U} + (\lambda - 1)A_{\tau_U}^{\rho^+}\right)
   h(X_{\tau_U})\right],
   \quad \forall x \in U.
   \]
  \end{theorem}
    \begin{proof}
     If $h$ is $\lambda$-harmonic on $B(z_1, r_1)$ and $B(z_2, r_1)$, then
     $h$ is $\lambda$-harmonic on $B(z_1, r_1) \cup B(z_2, r_2)$ by Lemma
     4.4. For any $z \in \bR^d$, there exists $r > 0$
     such that $h$ is $\lambda$-harmonic on $B(z,r)$ by Theorem \ref{T:HF}. 
     Since such $\{ B(z_i, r_i)\}$ covers $U$, we obtain this theorem. 
    \end{proof}

    \begin{corollary}
     \la{cor-main}
     If $\lambda = 1$, the harmonic function $h$ of $\cH^{\mu,F}$
     satisfies for any relatively compact open subset $U$ in $\bR^d$, 
     \begin{align}
      h(x) = \bE_x\left[\exp\left(A_{\tau_U}\right) h(X_{\tau_U})\right],
      \quad \text{for any } x \in U.
     \end{align}
    \end{corollary}

\bigskip

\textsc{Department of Mathematics, National Defense Academy, Yokosuka,
2390--8686, Japan}

{\it Email address:} tsuchida@nda.ac.jp

\end{document}